\documentclass[11pt]{article}
\usepackage{amsmath,amssymb,amsthm,amscd}
\usepackage{xcolor}
\usepackage[all]{xy}

\usepackage{graphics}

\definecolor{grn}{rgb}{0,0.3,0}
\newtheorem{theorem}{Theorem}[section]
\newtheorem{lemma}[theorem]{Lemma}
\newtheorem{proposition}[theorem]{Proposition}

\theoremstyle{plain}

\theoremstyle{definition}

\numberwithin{equation}{section}

\renewcommand{\theenumi}{(\roman{enumi})}
\renewcommand{\labelenumi}{\textup{(\theenumi)}}

\title{
 Topological entropy in continuous orbit equivalence 
 of one-sided topological Markov shifts \\
}
\author{Kengo Matsumoto \\
Department of Mathematics \\
Joetsu University of Education \\
Joetsu, Niigata, 943-8512  JAPAN
\and
Hiroki Matui \\
Graduate School of Mathematical Sciences \\
The University of Tokyo \\
3-8-1 Komaba, Tokyo, 153-8914 JAPAN}

\begin{document}

\maketitle

\date{}

\def\det{{{\operatorname{det}}}}

%\maketitle
\begin{abstract} 
In this paper, 
we prove that the continuous orbit equivalence class of a one-sided topological Markov shift
contains a one-sided topological Markov shift whose topological entropy 
is greater than an arbitrary prescribed positive real number,
and also contains a one-sided topological Markov shift
whose topological entropy is less than an arbitrary prescribed positive real number. 
\end{abstract}


  
%\tableofcontents

\newcommand{\Ker}{\operatorname{Ker}}
\newcommand{\sgn}{\operatorname{sgn}}
\newcommand{\Ad}{\operatorname{Ad}}
\newcommand{\ad}{\operatorname{ad}}
\newcommand{\orb}{\operatorname{orb}}
\newcommand{\rank}{\operatorname{rank}}
\newcommand{\BF}{\operatorname{BF}}
\newcommand{\topology}{\operatorname{top}}

\def\Re{{\operatorname{Re}}}
\def\det{{{\operatorname{det}}}}
\newcommand{\K}{\operatorname{K}}

\newcommand{\sqK}{\operatorname{K}\!\operatorname{K}}

\newcommand{\bbK}{\mathbb{K}}
\newcommand{\N}{\mathbb{N}}
\newcommand{\bbC}{\mathbb{C}}
\newcommand{\R}{\mathbb{R}}
\newcommand{\Rp}{{\mathbb{R}}^*_+}
\newcommand{\T}{\mathcal{T}}
\newcommand{\bbT}{\mathbb{T}}

\newcommand{\Z}{\mathbb{Z}}
\newcommand{\Zp}{{\mathbb{Z}}_+}
\def\AF{{{\operatorname{AF}}}}

\def\TorZ{{{\operatorname{Tor}}^\Z_1}}
\def\Ext{{{\operatorname{Ext}}}}
\def\Exts{\operatorname{Ext}_{\operatorname{s}}}
\def\Extw{\operatorname{Ext}_{\operatorname{w}}}
\def\Ext{{{\operatorname{Ext}}}}
\def\Free{{{\operatorname{Free}}}}

\def\OA{{{\mathcal{O}}_A}}
\def\ON{{{\mathcal{O}}_N}}
\def\OAT{{{\mathcal{O}}_{A^t}}}

\def\whatA{{\widehat{\A}}}
\def\whatB{{\widehat{\B}}}

\def\A{{\mathcal{A}}}
\def\B{{\mathcal{B}}}
\def\C{{\mathcal{C}}}
\def\D{{\mathcal{D}}}
\def\E{\mathcal{E}}
\def\F{{\mathcal{F}}}
\def\G{{\mathcal{G}}}
\def\O{{\mathcal{O}}}
\def\V{{\mathcal{V}}}

\def\OB{{{\mathcal{O}}_B}}
\def\OTA{{{\mathcal{O}}_{\tilde{A}}}}
\def\FA{{{\mathcal{F}}_A}}
\def\PA{{{\mathcal{P}}_A}}
\def\PI{{{\mathcal{P}}_\infty}}
\def\OI{{{\mathcal{O}}_\infty}}

\def\calI{\mathcal{I}}
\def\calK{\mathcal{K}}
\def\calP{\mathcal{P}}
\def\calQ{\mathcal{Q}}
\def\calR{\mathcal{R}}
\def\calC{\mathcal{C}}
\def\calD{\mathcal{D}}
\def\calM{\mathcal{M}}

\def\bbC{{\mathbb{C}}}

\def\U{{\mathcal{U}}}
\def\OF{{{\mathcal{O}}_F}}
\def\DF{{{\mathcal{D}}_F}}
\def\FB{{{\mathcal{F}}_B}}
\def\DA{{{\mathcal{D}}_A}}
\def\DB{{{\mathcal{D}}_B}}
\def\DZ{{{\mathcal{D}}_Z}}

\def\End{{{\operatorname{End}}}}

\def\Ext{{{\operatorname{Ext}}}}
\def\Hom{{{\operatorname{Hom}}}}

\def\Tor{{{\operatorname{Tor}}}}

\def\Max{{{\operatorname{Max}}}}

\def\max{{{\operatorname{max}}}}

\def\GCD{{{\operatorname{GCD}}}}

\def\KMS{{{\operatorname{KMS}}}}
\def\Per{{{\operatorname{Per}}}}
\def\Out{{{\operatorname{Out}}}}
\def\Aut{{{\operatorname{Aut}}}}
\def\Ad{{{\operatorname{Ad}}}}
\def\Inn{{{\operatorname{Inn}}}}
\def\Int{{{\operatorname{Int}}}}
\def\det{{{\operatorname{det}}}}
\def\exp{{{\operatorname{exp}}}}
\def\nep{{{\operatorname{nep}}}}
\def\sgn{{{\operatorname{sign}}}}
\def\cobdy{{{\operatorname{cobdy}}}}
\def\Ker{{{\operatorname{Ker}}}}
\def\Coker{{{\operatorname{Coker}}}}
\def\Im{{\operatorname{Im}}}

\def\top{{{\operatorname{top}}}}
\def\ind{{{\operatorname{ind}}}}
\def\Ind{{{\operatorname{Ind}}}}
\def\id{{{\operatorname{id}}}}
\def\supp{{{\operatorname{supp}}}}
\def\co{{{\operatorname{co}}}}
\def\scoe{{{\operatorname{scoe}}}}
\def\coe{{{\operatorname{coe}}}}
\def\COE{{{\operatorname{COE}}}}
\def\I{{\mathcal{I}}}
\def\Span{{{\operatorname{Span}}}}

\def\event{{{\operatorname{event}}}}
\def\Proj{{{\operatorname{Proj}}}}
\def\S{\mathcal{S}}

\def\whatOA{\widehat{\mathcal{O}}_A}
\def\whatOAT{\widehat{\mathcal{O}}_{A^t}}

\def\wtO{\widetilde{O}}
\def\wtA{\widetilde{A}}

\def\opE{\operatorname{E}}

\def\Cokern{\Coker(I_n - \widetilde{A}_{C_n})}
\def\Cokernplus1{\Coker(I_{n+1} - \widetilde{A}_{C_{n+1}})}

\newcommand{\SN}{\operatorname{SN}}

\def\wA2{\widehat{A^{[2]}}}
\def\Tr{{{\operatorname{Tr}}}}
\def\Sp{{\operatorname{SP}}}
\def\Spx{{{\operatorname{Sp}}^\times}}
%\newpage

%%%%%%%%%%%%%%%%%%%%%%%%%%%%%%%%%%%%%%%

\bigskip

%%%%%%%%%%%%%%%%%%%%%%%%%%%%%%%%%%%%%%%%%%%%
\section{Introduction}
%%%%%%%%%%%%%%%%%%%%%
F. Sugisaki in \cite{SugisakiI} and \cite{SugisakiII}
proved that any Cantor minimal system is strongly orbit equivalent to Cantor minimal 
systems of all entropies (\cite{SugisakiI} for finite entropies,
\cite{SugisakiII} for infinite entropy). His results gave a complete answer 
for the conjecture raised by Boyle--Handelman \cite{BH1994} (cf. \cite{Ormes}).  
Inspired by Sugisaki's results, in this paper,
we study a relationship between topological entropy and continuous orbit equivalence 
of one-sided topological Markov shifts. 
Cantor minimal systems are minimal homeomorphisms of Cantor sets.
On the contrary, one-sided topological Markov shifts are surjective continuous maps having many periodic points.
Let $A=[A(i,j)]_{i,j=1}^N$ be an irreducible non-permutation matrix with entries in $\{0,1\}$.
The one-sided topological Markov shift $(X_A,\sigma_A)$ is defined to be a shift transformation
$\sigma_A$ on the shift space $X_A$
consisting of all right one-sided sequences $(x_n)_{n\in \N}$ of $x_n \in \{1,\dots,N\}$
satisfying $A(x_n, x_{n+1}) =1$  for all $n \in \N$. The shift $\sigma_A$ on $X_A$ is defined by
$\sigma_A((x_n)_{n \in \N}) = (x_{n+1})_{n \in \N}$.
By equipping $X_A$ with the relative topology of the infinite product topology on $\{1,\dots, N\}^\N$
of discrete set $\{1,\dots,N\}$, 
the shift space $X_A$ is homeomorphic to a Cantor discontinuum.
In \cite{MaPacific}, the notion of continuous orbit equivalence 
in one-sided topological Markov shifts was introduced in the following way.
Two one-sided topological Markov shifts
$(X_A,\sigma_A)$ and $(X_B,\sigma_B)$ are {\it continuously orbit equivalent,}\/ 
written as $(X_A, \sigma_A) \underset{\COE}{\sim}(X_B,\sigma_B)$,
if there  exist
a homeomorphism $h: X_A \to X_B$ and continuous maps 
$k_1, l_1: X_A \to \Zp=\{ 0,1,\dots,\}$ and 
$k_2, l_2: X_B \to \Zp$
satisfying 
\begin{align*}
\sigma_B^{k_1(x)}(h(\sigma_A(x))) = & \sigma_B^{l_1(x)}(h(x))\quad \text{ for } x \in X_A, \\
\sigma_A^{k_2(y)}(h^{-1}(\sigma_B(y))) = & \sigma_A^{l_2(y)}(h^{-1}(y))\quad \text{ for } y \in X_B.
\end{align*}
In \cite{MMKyoto},
the classification of continuous orbit equivalence was completed so that we know
$(X_A, \sigma_A) \underset{\COE}{\sim}(X_B,\sigma_B)$
if and only if the associated Cuntz--Krieger algebras $\OA, \OB$ are isomorphic and
$\det(I_N - A) = \det(I_M-B)$,
where
the matrix size of $A$ is $N$ and that of $B$ is $M$, and $I_N, I_M$ 
denote the identity matrices, respectively. 
Let us denote by 
$
\BF(A^t)
$
the quotient group
$\Z^N/(I_N -A^t)\Z^N$, which is well-known as the Bowen--Franks group for the matrix
$A^t$ (\cite{BF}, cf. \cite{Kitchens}, \cite{LM}, etc.).
Since 
the isomorphism class of $\OA$ is completely determined by 
its $\K$-group together with the position $[1_\OA]$ of the unit $1_{\OA}$ of $\OA$ 
in $\K_0(\OA)$ by  
\cite{Ro},
and $(\K_0(\OA), [1_\OA])\cong (\BF(A^t), [1_N])$ by  
\cite{Cuntz80},
we see that the triplet
$(\BF(A^t), [1_N], \det(I_N-A))$
is a complete set of invariants of the continuous orbit equivalence of $(X_A, \sigma_A)$,
where $[1_N]$ denotes the class of the vector $(1,\dots,1)$ in the quotient group
$\BF(A^t)$.

On the other hand, it is well-known as Parry's theorem \cite{Parry}
that the topological entropy $h_{\operatorname{top}}(X_A, \sigma_A)$
of a topological Markov shift $(X_A,\sigma_A)$ is computed to be $\log \lambda_A$, 
where $\lambda_A$ is the Perron--Frobenius eigenvalue of the underlying matrix $A$
(cf. \cite[Theorem 4.3.1]{LM}, \cite[Observation 1.4.2]{Kitchens})).

  In this paper, 
we prove that the continuous orbit equivalence class of a one-sided topological Markov shift
contains a one-sided topological Markov shift whose topological entropy 
is greater than an arbitrary prescribed positive real number (Theorem \ref{thm:main1}),
and also contains 
a one-sided topological Markov shift
whose topological entropy is less than an arbitrary prescribed positive real number 
(Theorem \ref{thm:main2}).  

In what follows, we denote by $\Zp$ and $\N$ the set of nonnegative integers
and the set of positive integers, respectively.  
%%%%%%%%%%%%%%%%%%%%%%%%%%%%%%%%%%%
\section{Upper Entropy}\label{sec:Upper}
%%%%%%%%%%%%%%%%%%%%%%%%%%%%%%%%%%
In this section,
we study upper  entropies in continuous orbit equivalence class.
  We have to provide a couple of notions and lemmas to show Theorem \ref{thm:main1}.
An essential matrix means a matrix for which none of its rows or columns is zero.

Let $A=[A(i,j)]_{i,j=1}^N$
be an $N\times N$ essential matrix  with entries in nonnegative integers.
A directed graph $\G_A=(\V_A, \E_A)$ naturally associates to the matrix $A$ in the following way.
The vertex set $\V_A$ is defined as $\{1,\dots,N\}$.
For two vertices $i,j \in \V_A$ with $A(i,j) \ne 0$,
define $A(i,j)$ multiple directed edges from $i$ to $j$. 
The set of such edges is the edge set $\E_A$.
If $\alpha \in \E_A$ is a directed edge from $i$ to $j$,
we write $s(\alpha) = i$ and $t(\alpha) =j$.   
We then have a matrix 
$A^{[2]}=[A^{[2]}(\alpha,\beta)]_{\alpha,\beta \in \E_A}$
with entries in $\{0,1\}$ such that 
$A^{[2]}(\alpha,\beta) =1$ if $t(\alpha) = s(\beta)$, otherwise $0$.
The matrix $A^{[2]}$ is called the second higher block matrix for $A$ (cf. \cite[Section 2.4]{LM}).
Define $\V_A\times \E_A$-matrix $D=[D(i,\alpha)]_{i \in \V_A, \alpha\in \E_A}$ and 
$\E_A\times \V_A$-matrix $E=[E(\beta,j)]_{\beta\in \E_A, j \in \V_A}$ by setting  for $i,j\in \V_A$ and $\alpha,\beta \in \E_A$
\begin{equation}\label{eq:DE}
D(i,\alpha)=
\begin{cases}
1 & \text{ if } i = s(\alpha), \\
0 & \text{ otherwise}
\end{cases}
\quad
\text{ and }
\quad
E(\beta,j)=
\begin{cases}
1 & \text{ if } j = t(\beta), \\
0 & \text{ otherwise,}
\end{cases}
 \end{equation}
so that $A= DE, \, A^{[2]} = ED$.
The matrix $D$ is called 
a division matrix which is a rectangular matrix with entries in $\{0,1\}$ and with exactly one $1$ 
in each column and at least one $1$ in each row, and
$E$ is called an amalgamation matrix  
which is a rectangular matrix with entries in $\{0,1\}$ and with exactly one $1$ 
in each row and at least one $1$ in each column.
Let $M$ be the cardinal number $|\E_A|$ of the edge set $\E_A$.
The following lemma is well-known (cf. \cite{LM}).
\begin{lemma}\label{lem:Outsplit}
\begin{enumerate}
\renewcommand{\theenumi}{(\roman{enumi})}
\renewcommand{\labelenumi}{\textup{\theenumi}}
\item
The map
$x \in \Z^N \to D^t x \in \Z^M$ induces an isomorphism
 $ \Phi_{A}:[x] \in \BF(A^t) \to [D^t x] \in \BF({A^{[2]}}^t)$
such that $\Phi_{A}([1_N]) = [1_M]$. 
\item $\det(I_N - A) = \det(I_M - A^{[2]})$ and $\lambda_A = \lambda_A^{[2]}$.
\end{enumerate}
Hence the Cuntz--Krieger algebras $\OA$ and $\O_{A^{[2]}}$ are isomorphic and
$(X_A, \sigma_A) \underset{\COE}{\sim}(X_{A^{[2]}},\sigma_{A^{[2]}})$.
\end{lemma}\label{lem:Abar}

We shall use a slight variant of the construction in \cite{MaMathScan2018}.
Although the result needed below follows directly from \cite{MaMathScan2018},
the following explicit form of the matrix $\Bar{A}$ 
is useful because it keeps the original
matrix $A$ as the upper-left block.  
This makes the later reduction to the case
where the matrix has a non-zero diagonal entry more transparent.

\begin{lemma}[{cf. \cite{MaMathScan2018}}]
For an irreducible non-permutation matrix $A=[A(i,j)]_{i,j=1}^N$
 with entries in $\{0,1\}$,
let $\Bar{A}$ be the $(N+3)\times (N+3)$ matrix defined by 
\begin{equation*}
\Bar{A}
=
%{\small 
\setcounter{MaxMatrixCols}{15}
{\small 
\begin{bmatrix}
A(1,1)  &\cdots&A(1,N-1)  &A(1,N)    &0          &0     &0 \\ 
\vdots  &      &\vdots    &\vdots    &\vdots     &\vdots &\vdots \\ 
A(N-1,1)&\cdots&A(N-1,N-1)&A(N-1,N)  &0          &0     & 0 \\ 
A(N,1)  &\cdots &A(N,N-1) &A(N,N)    &0          &1     &0 \\ 
0       &\cdots &0        &1          &0          &0     &0   \\ 
0       &\cdots &0        &0          &1          &0     &1  \\ 
0       &\cdots &0        &0          &0          &1     &1    
\end{bmatrix}
}
=
\setcounter{MaxMatrixCols}{15}
{\small 
\begin{bmatrix}
         &       &           &           &0          &0     &0 \\ 
         &       &           &           &\vdots     &\vdots&\vdots \\ 
         &\text{\huge{A}}&           &           &0          &0     & 0 \\   
         &       &           &           &0          &1     &0 \\ 
0        &\cdots &0          &1          &0          &0     &0   \\ 
0        &\cdots &0          &0          &1          &0     &1  \\ 
0        &\cdots &0          &0          &0          &1     &1    
\end{bmatrix}.
}
\end{equation*}
Then there exists an isomorphism 
$\Psi: \BF({\Bar{A}}^t) \to \BF({A}^t)
$
such that 
\begin{equation}\label{eq:Psi}
\Psi([1_{N+3}]) = [1_N]
\quad 
\text{  and }
\quad 
\det(I_{N+3}-\Bar{A}) = - \det(I_N - A).
\end{equation}
\end{lemma}
\begin{proof}
The pair $(I_{N+3}-\Bar{A}^t, [1_{N+3}])$ of the matrix 
$I_{N+3}-\Bar{A}^t$ and the vector $[1_{N+3}]$ 
is transformed by elementary operations of matrices 
in the following way:
 {\allowdisplaybreaks
\begin{align*}
 &(I_{N+3}-\Bar{A}^t, [1_{N+3}])  \\
 = & 
\left(
{\small 
\begin{bmatrix}
         &       &           &           &0          &0     &0 \\ 
         &       &           &           &\vdots     &\vdots&\vdots \\ 
         &\text{\huge{I}}_{N}\text{\huge{-}}\text{\huge{A}}^{\text{\large{t}}}  &           &           &0          &0     & 0 \\   
         &       &           &           &-1         &0     &0 \\ 
0        &\cdots &0          &0          &1          &-1    &0   \\ 
0        &\cdots &0          &-1         &0          &1     &-1 \\ 
0        &\cdots &0          &0          &0          &-1    &0    
\end{bmatrix},
\begin{bmatrix}
1 \\
\vdots \\
1 \\
1 \\
1 \\
1 \\
1
\end{bmatrix}
}
\right) (\text{add } (N+3)\text{th column to } (N+2)\text{th column}) \\
\to &
\left(
{\small 
\begin{bmatrix}
         &       &           &           &0          &0     &0 \\ 
         &       &           &           &\vdots     &\vdots&\vdots \\ 
         &\text{\huge{I}}_{N}\text{\huge{-}}\text{\huge{A}}^{\text{\large{t}}}  &           &           &0          &0     & 0 \\   
         &       &           &           &-1         &0     &0 \\ 
0        &\cdots &0          &0          &1          &-1    &0   \\ 
0        &\cdots &0          &-1         &0          &0     &-1 \\ 
0        &\cdots &0          &0          &0          &-1    &0    
\end{bmatrix},
\begin{bmatrix}
1 \\
\vdots \\
1 \\
1 \\
1 \\
1 \\
1
\end{bmatrix}
}
\right) (\text{subtract } (N+3)\text{th column from } N\text{th column}) \\
\to &
\left(
{\small 
\begin{bmatrix}
         &       &           &           &0          &0     &0 \\ 
         &       &           &           &\vdots     &\vdots&\vdots \\ 
         &\text{\huge{I}}_{N}\text{\huge{-}}\text{\huge{A}}^{\text{\large{t}}} &           &           &0          &0     & 0 \\   
         &       &           &           &-1         &0     &0 \\ 
0        &\cdots &0          &0          &1          &-1    &0   \\ 
0        &\cdots &0          &0          &0          &0     &-1 \\ 
0        &\cdots &0          &0          &0          &-1    &0    
\end{bmatrix},
\begin{bmatrix}
1 \\
\vdots \\
1 \\
1 \\
1 \\
1 \\
1
\end{bmatrix}
}
\right) (\text{add } (N+1)\text{th row to } N\text{th row}) \\
\to &
\left(
{\small 
\begin{bmatrix}
         &      &           &           &0          &0     &0 \\ 
         &      &           &           &\vdots     &\vdots&\vdots \\ 
         &\text{\huge{I}}_{N}\text{\huge{-}}\text{\huge{A}}^{\text{\large{t}}} &           &           &0          &0     & 0 \\   
         &      &           &           &0          &-1    &0 \\ 
0        &\cdots &0          &0          &1          &-1    &0   \\ 
0        &\cdots &0          &0          &0          &0     &-1 \\ 
0        &\cdots &0          &0          &0          &-1    &0    
\end{bmatrix},
\begin{bmatrix}
1 \\
\vdots \\
1 \\
2 \\
1 \\
1 \\
1
\end{bmatrix}
}
\right) (\text{subtract } (N+3)\text{th row from } N\text{th row}) \\
\to &
\left(
{\small 
\begin{bmatrix}
         &       &           &           &0          &0     &0 \\ 
         &       &           &           &\vdots     &\vdots&\vdots \\ 
         &\text{\huge{I}}_{N}\text{\huge{-}}\text{\huge{A}}^{\text{\large{t}}}  &           &           &0          &0     & 0 \\   
         &       &           &           &0          &0     &0 \\ 
0        &\cdots &0          &0          &1          &-1    &0   \\ 
0        &\cdots &0          &0          &0          &0     &-1 \\ 
0        &\cdots &0          &0          &0          &-1    &0    
\end{bmatrix},
\begin{bmatrix}
1 \\
\vdots \\
1 \\
1 \\
1 \\
1 \\
1
\end{bmatrix}
}
\right)(\text{add } (N+1)\text{th column to  } (N+2)\text{th column}) \\
\to &
\left(
{\small 
\begin{bmatrix}
         &       &           &           &0          &0     &0 \\ 
         &       &           &           &\vdots     &\vdots&\vdots \\ 
         &\text{\huge{I}}_{N}\text{\huge{-}}\text{\huge{A}}^{\text{\large{t}}} &           &           &0          &0     & 0 \\   
         &       &           &           &0          &0     &0 \\ 
0        &\cdots &0          &0          &1          &0     &0   \\ 
0        &\cdots &0          &0          &0          &0     &-1 \\ 
0        &\cdots &0          &0          &0          &-1    &0    
\end{bmatrix},
\begin{bmatrix}
1 \\
\vdots \\
1 \\
1 \\
1 \\
1 \\
1
\end{bmatrix}
}
\right)(\text{exchange } (N+2)\text{th row and } (N+3)\text{th row}) \\
\to &
\left(
{\small 
\begin{bmatrix}
         &       &           &           &0          &0     &0 \\ 
         &       &           &           &\vdots     &\vdots&\vdots \\ 
         &\text{\huge{I}}_{N}\text{\huge{-}}\text{\huge{A}}^{\text{\large{t}}} &           &           &0          &0     & 0 \\   
         &       &           &           &0          &0     &0 \\ 
0        &\cdots &0          &0          &1          &0     &0   \\ 
0        &\cdots &0          &0          &0          &-1    &0  \\ 
0        &\cdots &0          &0          &0          &0     &-1   
\end{bmatrix},
\begin{bmatrix}
1 \\
\vdots \\
1 \\
1 \\
1 \\
1 \\
1
\end{bmatrix}
}
\right).
\end{align*}
}
By the above operations of matrices, we see that 
 there exists an isomorphism 
$\Psi: \BF({\Bar{A}}^t) \to \BF({A}^t)
$
satisfying \eqref{eq:Psi}.
Since all the operations except the final exchange of two rows preserve the determinant,
and the final exchange changes its sign, the determinant formula in \eqref{eq:Psi} follows.
\end{proof}
Therefore we have the following lemma.
\begin{lemma}\label{lem:doubleBarA}
For an irreducible non-permutation matrix $A=[A(i,j)]_{i,j=1}^N$
 with entries in $\{0,1\}$,
let $\Bar{\Bar{A}}$ be the $(N+6)\times (N+6)$ matrix  $({\Bar{A}})^{\bar{}}$ defined by 
\begin{align*}
\Bar{\Bar{A}} 
= &
{\small 
\setcounter{MaxMatrixCols}{15}
\begin{bmatrix}
A(1,1)   &\cdots &A(1,N-1)  &A(1,N)    &0          &0     &0     &0          &0     &0\\ 
\vdots   &       &\vdots    &\vdots    &\vdots     &\vdots&\vdots&\vdots     &\vdots&\vdots \\
A(N-1,1) &\cdots &A(N-1,N-1)&A(N-1,N)  &0          &0     & 0    &0          &0     &0\\ 
A(N,1)   &\cdots &A(N,N-1)  &A(N,N)    &0          &1     &0     &0          &0     &0 \\
0        &\cdots &0         &1         &0          &0     &0     &0          &0     &0 \\
0        &\cdots &0         &0         &1          &0     &1     &0          &0     &0 \\
0        &\cdots &0         &0         &0          &1     &1     &0          &1     &0 \\
0        &\cdots &0         &0         &0          &0     &1     &0          &0     &0 \\
0        &\cdots &0         &0         &0          &0     &0     &1          &0     &1 \\
0        &\cdots &0         &0         &0          &0     &0     &0          &1     &1   
\end{bmatrix} }\\
= &
{\small 
\setcounter{MaxMatrixCols}{15}
\begin{bmatrix}
         &       &           &           &0          &0     &0 \\ 
         &       &           &           &\vdots     &\vdots&\vdots \\ 
         &\text{\huge{$\Bar{A}$}}&  &           &0          &0     &0 \\   
         &       &           &           &0          &1     &0 \\ 
0        &\cdots &0          &1          &0          &0     &0   \\ 
0        &\cdots &0          &0          &1          &0     &1  \\ 
0        &\cdots &0          &0          &0          &1     &1    
\end{bmatrix}
}
\end{align*}
Then there exists an isomorphism 
$\Phi: \BF({\Bar{\Bar{A}}}^t) \to \BF({A}^t)
$
%$\Phi: \Z^{N+6}/(I_{N+6}- B^t)\Z^{N+6} \to \Z^{N}/(I_{N}-{A}^t)\Z^{N}$
such that 
\begin{equation*}
\Phi([1_{N+6}]) = [1_N]
\quad 
\text{  and }
\quad 
\det(I_{N+6}- {\Bar{\Bar{A}}}) = \det(I_N - A).
\end{equation*}
Hence we have 
%\begin{equation*}
$(X_{\Bar{\Bar{A}}}, \sigma_{\Bar{\Bar{A}}}) \underset{\COE}{\sim} (X_A,\sigma_A). $
%\end{equation*}
\end{lemma}
We note that there exist nonzero diagonal entries of the above matrix ${\Bar{\Bar{A}}}$.
\begin{proposition}\label{prop:main1}
For an irreducible non-permutation matrix $A$ with entries in $\{0,1\}$
and a positive integer $n \in \N$, 
there exists  an irreducible non-permutation matrix $B$ with entries in $\{0,1\}$
such that 
\begin{equation*}
(X_A, \sigma_A) \underset{\COE}{\sim} (X_B,\sigma_B) \quad \text{ and }\quad
\lambda_B \geq n 
\end{equation*}
where 
$\lambda_B$ is the  Perron--Frobenius eigenvalue of $B$. 
\end{proposition}
\begin{proof}
Let $A=[A(i,j)]_{i,j=1}^N$ be an irreducible non-permutation matrix  with entries in $\{0,1\}$.
By considering $\Bar{\Bar{A}}$ in Lemma \ref{lem:doubleBarA} instead of $A$, 
we may assume that  $A(p,p) =1$ for some $p =1,\dots,N$.
Put the matrix $\widetilde{A}:= A -I_N$, so that $\widetilde{A}(p,p) =0.$
Take an arbitrary fixed positive integer $n$.  
Since $A$ is irreducible, there exists $q\ne p$ such that $A(p,q) =1$ and hence $\widetilde{A}(p,q) =1$.
Let $\widetilde{A}^{\prime}$ be the $N \times N$ matrix obtained from $\widetilde{A}$
by adding $n$ times $p$th row to the $q$th row. 
The matrix $\widetilde{A}^{\prime}$ is of the form
\begin{equation*}
\widetilde{A}^{\prime} = P\cdot \widetilde{A}
\end{equation*}  
for the $N\times N$ invertible matrix $P=[P(i,j)]_{i,j=1}^N\in GL(N,\Z)$ defined by
\begin{equation*}
P(i,j) = 
\begin{cases}
1 & \text{ if } i=j, \\
n & \text{ if } (i, j) = (q,p),\\
0 & \text{ otherwise.}
\end{cases}
\end{equation*}
The matrix  
 $\widetilde{A}^{\prime} + I_N$ is an irreducible non-permutation matrix
 having its entries in nonnegative integers,
 because $(\widetilde{A}^{\prime} + I_N)(i,j) \ge A(i,j)$ for all $i,j=1,\dots,N$.
Define the matrix $B$ to be the second higher block matrix $(\widetilde{A}^{\prime}+I_N)^{[2]}$
of $\widetilde{A}^{\prime} +I_N$.
Let $M$ denote the size of the matrix $B$, so that $B$ is an $M\times M$ irreducible non-permutation matrix with entries in $\{0,1\}$. 
By using Lemma \ref{lem:Outsplit}, we have an isomorphism
$\Phi: \BF(A^t) \to \BF(B^t)$
such that $\Phi([1_N]) = [1_M]$.
Since $B$ is the second higher block matrix of 
$\widetilde{A}^\prime + I_N$, 
we have
$\det(I_M -B) =\det(I_N -(\widetilde{A}^\prime + I_N)).$
By noting 
 $\det(P) = 1$, the equalities 
\begin{equation*}
\det(I_N-A) 
=\det(-\widetilde{A}) 
= \det(-\widetilde{A}^\prime) 
= \det(I_N -(\widetilde{A}^\prime + I_N)) 
=\det(I_M - B) 
\end{equation*}
hold.
We thus conclude that 
$(X_B, \sigma_B) \underset{\COE}{\sim} (X_A,\sigma_A).$
Let us denote by $\lambda_B$ and $\lambda_{\widetilde{A}^\prime + I_N}$ 
the Perron--Frobenius  eigenvalues of the matrices 
$B$ and $\widetilde{A}^\prime + I_N$, respectively,
so that 
$\lambda_B = \lambda_{\widetilde{A}^\prime + I_N}$.
As
$(\widetilde{A}^{\prime} + I_N)(q,q) \ge n$,
the subgraph with the single vertex $q$ has at least $n$ multiple loops, so that  
$
\lambda_{\widetilde{A}^{\prime} + I_N} \ge n 
$
 and hence
 $ 
\lambda_B \ge n.
$
We therefore end the proof of Proposition \ref{prop:main1}.
\end{proof}
It is well-known as Parry's theorem \cite{Parry} that the topological entropy $h_{\operatorname{top}}(X_A, \sigma_A)$
of a topological Markov shift $(X_A,\sigma_A)$ is computed to be $\log \lambda_A$
(cf. \cite[Theorem 4.3.1]{LM}). 
We thus obtain the following theorem.
\begin{theorem}\label{thm:main1}
For an irreducible non-permutation matrix $A$ with entries in $\{0,1\}$,
and $R>0$, 
there exists  an irreducible non-permutation matrix $B$ with entries in $\{0,1\}$
such that 
$$
(X_A, \sigma_A) \underset{\COE}{\sim} (X_B,\sigma_B) \quad \text{ and }\quad
h_{\operatorname{top}}(X_B, \sigma_B) >R. 
$$
\end{theorem}

%%%%%%%%%%%%%%%%%%%%%%%%%%%%%%%%%%%%%%%%%%%
\section{Lower entropy}\label{sec:Lower}
%%%%%%%%%%%%%%%%%%%%%%%%%%%%%%%%%%%%%%%%%%%%%
%%%%%%%%%%%%%%%%%%%%%%%%%%%%%%%%%%%%%%%%%%%%%

In this section,
we study lower entropies in a continuous orbit equivalence class.
In what follows, we prove that 
 a one-sided topological Markov shift 
$(X_A, \sigma_A)$ is continuously orbit equivalent to  
a topological Markov shift  $(X_B, \sigma_B)$
whose topological entropy $h_{\top}(X_B, \sigma_B)$
is less than an arbitrary  prescribed positive real number. 
Throughout the section, 
let $A=[A(i,j)]_{i,j=1}^N$ 
be an irreducible non-permutation matrix with entries in $\{0,1\}$.
\begin{lemma}\label{lem:A^{<m>}}
For $m \in \N$, define the $Nm \times Nm$-matrix $A^{<m>}$ 
as the block matrix:
\begin{equation}\label{eq:matrixA^{<m>}}
A^{<m>}: =
\begin{bmatrix}
0     & I_N  & 0    &\cdots& 0 \\
0     & 0    & I_N  &\ddots& \vdots \\
\vdots&\ddots&\ddots&\ddots& 0 \\
0     &\cdots&0     &0     &I_N \\
A     &0     &\cdots&0     & 0 
\end{bmatrix},
\end{equation}
where $0$s in the above matrix denote the $N\times N$ zero matrices.
Then we have
\begin{enumerate}
\renewcommand{\theenumi}{(\roman{enumi})}
\renewcommand{\labelenumi}{\textup{\theenumi}}
\item $(\BF(A^{<m> t}), [1_{Nm}]) \cong
       (\BF(A^t), m [1_N]).$
\item $\det(I_{Nm} - A^{<m>}) = \det(I_N - A).$
\item $\lambda_{A^{<m>}} = \lambda_A^{\frac{1}{m}}.$
\end{enumerate}
\end{lemma}
\begin{proof}
(i), (ii)
We have the following sequence of elementary operations of matrices:
{\allowdisplaybreaks
\begin{align*}
& (I_{Nm}- A^{<m> t}, [1_{Nm}]) \\
= &
\left(
\begin{bmatrix}
I_N   & 0    &\cdots&     0& -A^t\\
-I_N  & I_N  & 0    &\cdots& 0 \\
0     &\ddots&\ddots&\ddots& \vdots \\
\vdots&\ddots&-I_N  &I_N   &0   \\
0     &\cdots& 0    &-I_N  & I_N
\end{bmatrix}, \,
\begin{bmatrix}
1_N \\
1_N \\
\vdots \\
\vdots \\
1_N \\
\end{bmatrix}
\right) \text{ (add $m$th column to $(m-1)$th column)} \\
\to &
\left(
\begin{bmatrix}
I_N   & 0    &\cdots&-A^t  & -A^t\\
-I_N  & I_N  & 0    &\cdots& 0 \\
0     &\ddots&\ddots&\ddots& \vdots \\
\vdots&\ddots&-I_N  &I_N   &0   \\
0     &\cdots & 0    &0     & I_N
\end{bmatrix}, \,
\begin{bmatrix}
1_N \\
1_N \\
\vdots \\
\vdots \\
1_N \\
\end{bmatrix}
\right) \text{ (add $(m-1)$th column to $(m-2)$th column)} \\
\to & \cdots \\
\to &
\left(
\begin{bmatrix}
I_N-A^t&-A^t &\cdots&-A^t  & -A^t\\
0     & I_N  & 0    &\cdots& 0 \\
0     &\ddots&\ddots&\ddots& \vdots \\
\vdots&\ddots& 0    &I_N   & 0 \\
0     &\cdots& 0    &0     & I_N
\end{bmatrix}, \,
\begin{bmatrix}
1_N \\
1_N \\
\vdots \\
\vdots \\
1_N \\
\end{bmatrix}
\right) \text{ (add $A^t\times m$th row to $1$th row)}\\
\to &
\left(
\begin{bmatrix}
I_N-A^t&-A^t &\cdots&-A^t  & 0\\
0     & I_N  & 0    &\cdots& 0 \\
0     &\ddots&\ddots&\ddots& \vdots \\
\vdots&\ddots& 0    &I_N   & 0 \\
0     &\cdots& 0    &0     & I_N
\end{bmatrix}, \,
\begin{bmatrix}
1_N + A^t1_N\\
1_N \\
\vdots \\
\vdots \\
1_N \\
\end{bmatrix}
\right)  \text{ (add $A^t\times (m-1)$th row to $1$th row)}\\
\to & \cdots \\
\to &
\left(
\begin{bmatrix}
I_N-A^t&0    &\cdots&0    & 0\\
0     & I_N  & 0    &\cdots& 0 \\
0     &\ddots&\ddots&\ddots& \vdots \\
\vdots&\ddots& 0    &I_N   & 0 \\
0     &\cdots& 0    &0     & I_N
\end{bmatrix}, \,
\begin{bmatrix}
1_N +(m-1)A^t 1_N\\
1_N \\
\vdots \\
\vdots \\
1_N \\
\end{bmatrix}
\right).
\end{align*}
}
Hence we have
\begin{align*}
      (\BF(A^{<m> t}), [1_{Nm}]) 
\cong  (\BF(A^t), [1_N + (m-1)A^t 1_N]) 
\cong  (\BF(A^t), m [1_N]).
\end{align*}
The above matrix operations also shows that 
$\det(I_{Nm} - A^{<m>}) = \det(I_N - A).$

(iii)
Let $\lambda_{A^{<m>}}$ be the Perron--Frobenius eigenvalue of $A^{<m>}$.
One may find a sequence of positive vectors $v_i\in {\mathbb{R}}^N,\, i=1,\dots,m$
such that 
$$
A^{<m>}
\begin{bmatrix}
v_1 \\
v_2 \\
\vdots \\
v_m
\end{bmatrix} 
=\lambda_{A^{<m>}}
\begin{bmatrix}
v_1 \\
v_2 \\
\vdots \\
v_m
\end{bmatrix}
\quad 
\text{ so that }
\quad
\begin{bmatrix}
v_2 \\
v_3 \\
\vdots \\
v_m\\
Av_1
\end{bmatrix} 
=
\begin{bmatrix}
\lambda_{A^{<m>}} v_1 \\
\lambda_{A^{<m>}} v_2 \\
\vdots \\
\lambda_{A^{<m>}} v_{m-1} \\
\lambda_{A^{<m>}} v_m
\end{bmatrix}. 
$$
This shows that  
$$
Av_1 = (\lambda_{A^{<m>}})^m v_1.
$$
Hence $v_1$ is a positive eigenvector of $A$  for the positive eigenvalue 
$(\lambda_{A^{<m>}})^m$ which is the Perron-Frobenius eigenvalue $\lambda_A$ for the matrix $A$,
because of the uniqueness of the Perron-Frobenius eigenvalue.
\end{proof}

By using the above lemma, we have the following proposition.
\begin{proposition}\label{prop:A^{<m>}}
Let $A=[A(i,j)]_{i,j=1}^N$ 
be an irreducible non-permutation matrix with entries in $\{0,1\}$.
Assume that $[1_N]$ is a torsion element in the quotient group
$\BF(A^t)$.
Then there exists a sequence $A_n, \, n \in \N$ 
of irreducible non-permutation matrices with entries in $\{0,1\}$
such that
\begin{equation}\label{eq:theorem2}
\begin{cases}
\bullet & (X_A,\sigma_A) \underset{\COE}{\sim} (X_{A_n}, \sigma_{A_n}), \quad n \in \N, \\
\bullet & \lambda_{A_n} \downarrow 1.
\end{cases}
\end{equation}
\end{proposition}
\begin{proof}
Assume that $[1_N]$ is a torsion element in the quotient group
$\BF(A^t)$.
Hence one may find $\ell \in \N$ such that $ \ell [1_N] =0$ in $\Z^N/( I_N - A^t) \Z^N$.
For any $n \in \N$, put
$m_n = 1 + n \ell$, so that we have
\begin{equation*}
m_n[1_N] = [1_N] + n \ell[1_N] = [1_N].
\end{equation*}
Lemma \ref{lem:A^{<m>}} tells us that 
\begin{equation*}
(\BF(A^{<m_n> t}), [1_{Nm_n}]) 
\cong  (\BF(A^t), m_n [1_N])
\cong  (\BF(A^t), [1_N]).
\end{equation*}
Since $\det(I_{Nm_n} - A^{<m_n>}) = \det(I_N - A),$
by putting $A_n :=  A^{<m_n>}$, we have
$$
(X_A,\sigma_A) \underset{\COE}{\sim} (X_{A_n}, \sigma_{A_n}), \quad n \in \N.
$$
Since
$\lambda_{A^{<m_n>}} = \lambda_A^{\frac{1}{m_n}},$
we get 
$\lambda_{A_n} \downarrow 1.$
\end{proof}

%%%%%%%%%%%%%%%%%%%%%%%%%

The above proof does not work in case that $[1_N]$ is non-torsion in $\BF(A^t)$, 
so we have to provide several lemmas to treat general cases. 
We use the operation in matrices called the Parry-Sullivan move (\cite{PS}).
Let $A=[A(i,j)]_{i,j=1}^N$ 
be an irreducible non-permutation matrix with entries in $\{0,1\}$.
Take an arbitrary fixed $p \in \{1,\dots,N\}$,
and put
the new vertex set $\Sigma_{[p]} = \{1,\dots, N \} \cup \{p'\}$.
Define the matrix 
$A_{[p]} =[A_{[p]}(i,j)]_{i,j \in \Sigma_{[p]}}$ over $\Sigma_{[p]}$ by
setting
\begin{equation*}
A_{[p]}(i,j)
=
\begin{cases}
A(i,j) & \text{ for }  i\ne p, p'  \text{ and } j\ne p',\\
A(i,j) & \text{ for }  i=p' \text{ and } j\ne p', \\
1      & \text{ for }  i=p \text{ and } j=p', \\
0      & \text{ for }  i=p \text{ and } j\ne p',  \\
0      & \text{ for }  i \ne p \text{ and } j =p',
\end{cases}
\end{equation*}
so that 
\begin{equation*}
A_{[p]}
=
\begin{bmatrix}
A(1,1)  &\cdots & A(1,p) & 0    & A(1,p+1) &\cdots & A(1,N) \\ 
\vdots  &       & \vdots &\vdots&\vdots    &       & \vdots \\
A(p-1,1)&\cdots &A(p-1,p)& 0    &A(p-1,p+1)&\cdots &A(p-1,N) \\ 
0       &\cdots &0       & 1    &0         &\cdots &0        \\ 
A(p,1)  &\cdots &A(p,p)  & 0    &A(p,p+1)  &\cdots &A(p,N)   \\ 
\vdots  &       & \vdots &\vdots&\vdots    &       & \vdots \\ 
A(N,1)  &\cdots & A(N,p) & 0    & A(N,p+1) &\cdots & A(N,N) 
\end{bmatrix}.
\end{equation*}

\begin{lemma}\label{lem:PSmove1}
$\lambda_{A_{[p]}} < \lambda_A$.
\end{lemma}
\begin{proof}
Let $v=[v_i]_{i=1}^N$ be a positive eigenvector for the Perron-Frobenius eigenvalue 
$\lambda_A$
so that $Av =\lambda_Av$.
We then have the equality:
\begin{equation*}
{\footnotesize
\begin{bmatrix}
A(1,1)  &\cdots & A(1,p) & 0         & A(1,p+1) &\cdots & A(1,N) \\ 
\vdots  &       & \vdots &\vdots     &\vdots    &       & \vdots \\
A(p-1,1)&\cdots &A(p-1,p)& 0         &A(p-1,p+1)&\cdots &A(p-1,N) \\ 
0       &\cdots &0       & 1         &0         &\cdots &0        \\ 
A(p,1)  &\cdots &A(p,p)  &\lambda_A-1&A(p,p+1)  &\cdots &A(p,N)   \\ 
A(p+1,1)&\cdots &A(p+1,p)&0          &A(p+1,p+1)&\cdots &A(p+1,N)\\ 
\vdots  &       & \vdots &\vdots     &\vdots    &       & \vdots \\ 
A(N,1)  &\cdots & A(N,p) & 0         & A(N,p+1) &\cdots & A(N,N) 
\end{bmatrix}
\begin{bmatrix}
v_1 \\ 
\vdots \\
v_{p-1} \\ 
v_p    \\ 
\lambda_A v_{p} \\ 
v_{p+1}\\
\vdots \\ 
v_{N} 
\end{bmatrix}
=
\lambda_A
\begin{bmatrix}
v_1 \\ 
\vdots \\
v_{p-1} \\ 
v_p     \\ 
\lambda_Av_{p} \\ 
v_{p+1}\\
\vdots \\ 
v_{N} 
\end{bmatrix}.
}
\end{equation*}
Let us denote by $A'_{[p]}$ 
the above $(N+1)\times (N+1)$ matrix in the left hand side.
It differs from $A_{[p]}$ only at $(p+1,p+1)$-component.
We note that both of the matrices 
$A_{[p]}$ and $A'_{[p]}$ are irreducible, because so is the matrix $A$.
By the above equality, we see that $\lambda_A$ 
is the Perron-Frobenius eigenvalue of $A'_{[p]}$,
because  the vector in the right-hand is a positive vector.
As $0< \lambda_A -1 $ and hence $A_{[p]}< A'_{[p]}$, 
we know that  $\lambda_{A_{[p]}} < \lambda_{A'_{[p]}} =\lambda_A$ by
\cite[Theorem 1.5 (c)]{Seneta}.
\end{proof}
It is clear that $\det(I_{N+1}-A_{[p]})=\det(I_N-A)$. 

The proof of the following lemma is essentially seen in the proof of 
\cite[Theorem 1.3]{BF}, so we omit its proof.
\begin{lemma}\label{lem:PSmove2}
The  map
$\eta: \Z^{N+1}\to \Z^N$ 
defined by 
\begin{equation*}
\eta([x_1,\dots, x_{p-1}, x_p, x_{p'}, x_{p+1},\dots, x_N]) 
=[x_1,\dots, x_{p-1}, x_{p}+x_{p'}, x_{p+1}, x_{p+2}, \dots, x_N]) 
\end{equation*}
induces an isomorphism 
$\bar{\eta}: \BF(A_{[p]}^t) \to \BF(A^t)$ so that 
\begin{equation*}
\eta([1_{N+1}]) 
=[(\overbrace{1,\dots,1}^{\text{$p-1$ times}}, 
2,\overbrace{1,\dots,1}^{\text{$N-p$ times}})] 
\end{equation*}
\end{lemma}

By using the above lemma repeatedly
together with
Lemma \ref{lem:A^{<m>}}
and Lemma \ref{lem:PSmove1},
 we have the following lemma.
\begin{lemma}\label{lem:A_{[p],k}}
For any $p \in \{1,\dots,N\}$ and $k\in \Zp$, 
there exists an $(N+k)\times (N+k)$ matrix $A_{[p],k}$ with entries in $\{0,1\}$ 
such that 
there exists an isomorphism
$\bar{\eta}: \BF(A_{[p],k}^t) \to \BF(A^t)$ such that
\begin{gather*}
\bar{\eta}([1_{N+k}]) 
=[(\overbrace{1,\dots,1}^{\text{$p-1$ times}}, 
k+1,\overbrace{1,\dots,1}^{\text{$N-p$ times}})], \\
\det(I_{N+k} - A_{[p],k}) = \det(I_N - A), \qquad
0< \lambda_{A_{[p],k}} < \lambda_A.
\end{gather*}
\end{lemma}

We provide  one more lemma.
\begin{lemma}\label{lem:unitadjustment}
Let $e_i \in \Z^N$ be the vector whose $i$th component is one, 
and the other components are zeros.
Let $[e_i]$ be the class of $e_i$ in $\BF(A^t)$.
Then for any element $u \in \BF(A^t)$, there exist positive integers 
$m_1, \dots, m_N \in \N$ such that 
\begin{equation}\label{eq:um_1m_n}
u = m_1[e_1] + \cdots +m_N [e_N].
\end{equation}
\end{lemma}
\begin{proof}
Let $p$ be the period of the matrix $A$. 
Then $A^p$ is permutation-similar to a direct sum of 
$p$ primitive matrices, say 
$B_1\oplus B_2 \oplus \cdots \oplus B_p$.
There exists $n\in \N$ such that every entry
of $B_i^n$ is not less than $2$ 
for all $i = 1,2,\dots,p$. 
For 
$e := e_1 + e_2 +\cdots+ e_N,$
we have
\begin{equation*}
0 =[({(A^t)}^{pn}- I_N)e] =[{(A^t)}^{pn} e -e ] \quad \text{ in } \BF(A^t),
\end{equation*}
and the vector ${(A^t)}^{pn} e -e $ 
is a linear combination of ${e_i}^\prime$ s
with positive coefficients. 
Hence for any $a_1 e_1 + a_2 e_2 +\cdots+ a_N e_N \in \Z^N$, 
by taking large enough
$k \in \N$ and adding $k({(A^t)}^{pn}e -e)$, 
we may find a vector in $\N^N$ within the same
equivalence class in $\BF(A^t)$.
\end{proof}

We thus have the following proposition.
\begin{proposition}\label{prop:main2}
Let $A=[A(i,j)]_{i,j=1}^N$ 
be an irreducible non-permutation matrix with entries in $\{0,1\}$.
For any positive real number $\epsilon$,
there exists an irreducible non-permutation matrix $B$ with entries in $\{0,1\}$
 such that
\begin{equation*} 
 (X_A,\sigma_A) \underset{\COE}{\sim} (X_B, \sigma_B)
 \quad
 \text{ and }
 \quad
 1 < \lambda_B <1+\epsilon.
\end{equation*}
\end{proposition}
\begin{proof}
For the matrix $A=[A(i,j)]_{i,j=1}^N$,
there exists an irreducible non-permutation matrix 
$A_\circ=[A_\circ(i,j)]_{i,j=1}^{N_\circ}$ with entries in $\{0,1\}$ 
such that 
\begin{equation*}
\BF(A_\circ^t) \cong \BF(A^t), \quad
\det(I_{N_\circ}-A_{\circ}) = \det(I_N - A), 
\quad [1_{N_\circ}] =0 \text{ in } \BF(A_\circ^t)
\end{equation*}
(cf. \cite[Lemma 3.7]{MMKyoto}). %so that 
%$
%(X_{A_\circ}, \sigma_{A_\circ}) 
%\underset{\COE}{\sim}
%(X_A, \sigma_A).
%$
For any $\epsilon>0$, by Lemma \ref{lem:A^{<m>}}, one may find $m \in \N$ such that 
\begin{equation*}
(X_{A_\circ^{<m>}}, \sigma_{A_\circ^{<m>}}) 
\underset{\COE}{\sim}
(X_{A_\circ}, \sigma_{A_\circ}),
\qquad
 1< \lambda_{A_\circ^{<m>}} <1+\epsilon,
\end{equation*}
because $m [1_{N_\circ}] = 0.$
Put $N_1 = m N_\circ$. 
As $\BF(A_\circ^{<m> t}) \cong \BF(A^t)$ 
Lemma \ref{lem:unitadjustment} tells us that
there exist $m_1,\dots, m_{N_1} \in \N$ such that 
\begin{equation*}
(\BF(A_\circ^{<m> t}), m_1[e_1] + \cdots + m_{N_1} [e_{N_1}])
\cong
(\BF(A^t), [1_N]).
\end{equation*}
By using Lemma \ref{lem:A_{[p],k}} at each $i=1,\dots, N_1$,
one may find an irreducible non-permutation matrix $B=[B(i,j)]_{i,j=1}^M$ with entries in $\{0,1\}$
such that 
\begin{equation*}
(\BF(B^t), [1_M]) \cong (\BF(A^t), [1_N]), 
\quad
\det( I_M - B) = \det(I_N - A), 
\quad
1<\lambda_B < \lambda_{A_\circ^{<m>}}, 
\end{equation*}
so that 
$(X_B, \sigma_B) 
\underset{\COE}{\sim}
(X_A, \sigma_A)
$ 
and
$ 1< \lambda_B <1+\epsilon.
$
\end{proof}

We thus reach the following theorem.
\begin{theorem}\label{thm:main2}
Let $A$ be an irreducible non-permutation matrix with entries in $\{0,1\}$.
For any positive real number $r$,
there exists an irreducible non-permutation matrix $B$ with entries in $\{0,1\}$
 such that
\begin{equation*} 
 (X_A,\sigma_A) \underset{\COE}{\sim} (X_B, \sigma_B)
 \quad
 \text{ and }
 \quad
 h_{\top}(X_B,\sigma_B) <r.
\end{equation*}
\end{theorem}

\medskip

%%%%%%%%%%%%%%%%%%%%%%%%%%%%%
\section{Full shifts}
%%%%%%%%%%%%%%%%%%%%%%%%%%%%%%%%
Let $A =[A(i,j)]_{i,j=1}^N$ be the full $N$-shift with $N >1$
whose entries $A(i,j), \, i,j=1,\dots,N$ are all $1$s.
%%%%%%%%%%%%%%%%%%%%%%%%%%%%%%%%%%%%%%%%%%%%%%%%%%%%%%%%%%%%%%%%%
%%%%%%%%%%%%%%%%%%%%%%%%%%%%%%%%%%%%%%%%%%%%%%%%%%%%%%%%%%%%%%%%%

\medskip

%%%%%%%%%%%%%%%%%%%%%%%%%%%%%%%
{\bf 1. Upper entropy for full shifts.}
%%%%%%%%%%%%%%%%%%%%%%%%%%%%%%%%%% 

Following the proof of Proposition \ref{prop:main1} in Section \ref{sec:Upper},
we concretely construct a one-sided topological Markov shift 
 $(X_{B},\sigma_{B})$ such that 
 $(X_A,\sigma_A) \underset{\COE}{\sim} (X_{B}, \sigma_{B})$
 and $h_{\top}(X_B,\sigma_B)$ is greater than a prescribed integer.
Let $m$ be an arbitrary fixed positive integer.
 Let $p =N$ and $q = N-1$, 
Put the matrix $\widetilde{A}= A -I_N$.
Let $\widetilde{A}^{\prime}$ be the $N \times N$ matrix obtained from $\widetilde{A}$
by adding $m$ times $p$th row to the $q$th row. 
The matrix $x I_N - (\widetilde{A}^{\prime}+I_N) $ is of the form
\begin{equation*}
x I_N - (\widetilde{A}^{\prime} + I_N)
=
{\small
\begin{bmatrix}
x-1   &-1    &-1     & \dots & -1   &-1 \\
-1    & x-1  &-1     & \dots & -1   &-1 \\
\vdots&\ddots& \ddots&\ddots &\vdots&\vdots \\
-1    &\cdots&-1     &x-1    &-1    &-1 \\
-(m+1)&\cdots&\cdots&-(m+1)&x-(m+1)&-1 \\
-1    &\cdots&\cdots &-1     &-1    &x-1 
\end{bmatrix}.
}
\end{equation*}
It is a direct computation to show that 
\begin{equation*}
\det(xI_N - (\widetilde{A}^{\prime} + I_N)) =
x^{N-2}\{x^2 -(m+N)x + m\}
\end{equation*}
and hence
\begin{equation*}
\lambda_{\widetilde{A}^{\prime} + I_N} =
\frac{1}{2}\{ m+N + \sqrt{(m+N)^2 -4m} \}.
\end{equation*}
Let $B$  be the second higher block matrix $(\widetilde{A}^{\prime}+I_N)^{[2]}$
of $\widetilde{A}^{\prime} +I_N$.
As $\lambda_B = \lambda_{\widetilde{A}^{\prime} + I_N}$ % > m + N-1$,
the one-sided topological Markov shift 
 $(X_{B},\sigma_{B})$ is continuously orbit equivalent to the full $N$-shift
  $(X_A,\sigma_A)$ such that $\lambda_B> m + N-1$.

\medskip

%%%%%%%%%%%%%%%%%%%%%%%%%%%%%%%%
{\bf 2. Lower entropy for full shifts.}
%%%%%%%%%%%%%%%%%%%%%

Let $n$ be an arbitrary fixed positive integer.
Put $m_n = 1 + n(N-1).$
Let $A^{<m_n>}$ be the matrix defined by \eqref{eq:matrixA^{<m>}}
for the full $N$-shift $A$ and the positive integer $m_n$.
It is a direct computation to show that 
\begin{equation*}
\det(xI_{Nm_n} - A^{<m_n>}) =x^{Nm_n -m_n}(x^{m_n} - N)
\end{equation*}
and hence
\begin{equation*}
\lambda_{A^{<m_n>}} = N^{\frac{1}{m_n}}
\end{equation*}
Put $B = A^{<m_n>}$.
The 
one-sided topological Markov shift 
 $(X_{B},\sigma_{B})$ is continuously orbit equivalent to the full $N$-shift
  $(X_A,\sigma_A)$ such that $\lambda_B =N^{\frac{1}{m_n}}$.

\bigskip

%%%%%%%%%%%%%%%%%%%%%%%%%%%%%%%%%%%%%%
\medskip

{\it Acknowledgments:}
K. Matsumoto is supported by JSPS KAKENHI Grant Number 24K06775.
H. Matui is supported by JSPS KAKENHI Grant Number 23K22397.

%The authors would like to thank anonymous referees for their careful reading
%of the manuscript and their helpful comments and suggestions.
%The author is supported by JSPS KAKENHI Grant Number 24K06775.

%%%%%%%%%%

\end{document}